\documentclass[12pt]{article}
\usepackage{graphicx}
\usepackage{amsthm,amsmath}
\usepackage[comma,authoryear]{natbib}

\def\support{\footnote{This research was partially supported by NSF grant DMS 0802082 }}

\def\Title#1{\begin{center} {\large #1 } \end{center}}
\def\Author#1{\begin{center}{ \sc #1} \end{center}}

\newenvironment{Abstract}{\begin{quotation} \noindent ABSTRACT. }{\end{quotation}}

\numberwithin{equation}{section}
\theoremstyle{plain}
\newtheorem{thm}{Theorem}[section]
\newtheorem{cor}[thm]{Corollary}
\newtheorem{lem}[thm]{Lemma}
\newtheorem{prop}[thm]{Proposition}
\theoremstyle{remark}
\newtheorem{rem}[thm]{Remark}
\theoremstyle{definition}

\begin{document}
\Title{NO FEEDBACK CARD GUESSING FOR TOP TO RANDOM SHUFFLES}
\Author{ Lerna Pehlivan \support}
\begin{Abstract}
Consider $n$ cards that are labeled $1$ through $n$ with $n$ an even integer. The cards are put face down and their ordering starts with card labeled $1$ on top through card labeled $n$ at the bottom. The cards are top to random shuffled $m$ times and placed face down on the table. Starting from the top the cards are guessed without feedback (i.e. whether the guess was correct or false and what the guessed card was) one at a time. For $m > 4n\log n+cn$ we find a guessing strategy that maximizes the expected number of correct guesses. 
\end{Abstract}
\def\thefootnote{\fnsymbol{footnote}}
\setcounter{footnote}{0}
%

\section{Introduction}
Assuming that we have $n$ cards, initially ordered $1$ at top through $n$ at bottom, applying one {\em top to random shuffle} to this deck means taking the top card and placing it back into the deck at position $i$ with probability $w_i$. 
\cite{diaconis92} found 
that, when $w_i = \frac{1}{n}$ for all $i$, the eigenvalues of the transition matrix that is associated with this Markov chain are $0,\frac{1}{n},\ldots,\frac{n-2}{n},1$ and the multiplicities of the eigenvalues that are of the form $\frac{i}{n}$ are the number of permutations with exactly $i$ fixed points.\\
A {\em random to top shuffle}, often referred to as {\em Tsetlin library scheme} or {\em move to front scheme}, is the inverse of a top to random shuffle. Considering that we have $n$ records, the move to front scheme is one of the most famous self-organizing rules that has been considered for searching a particular record with minimum cost. \cite{mccabe} proved the existence of the limiting average number of records to be examined for a request of information.
Phatarfod studied the random to top shuffles on a deck of $n$ cards where the $i^{th}$ card is selected with probability $w_i$. He determined that it is a
Markov chain, but not a random walk on a group and that the stationary distribution of this Markov chain is not uniform. \cite{phatar2} also computed the
$m^{th}$ step transition matrix for this Markov chain by exploiting a connection with the coupon collector's problem. \cite{phatar} proved that the non-zero
eigenvalues of the transition matrix for the random to top shuffle are real, nonnegative, of the form $w_i, w_i+w_j, \ldots, \sum_{i=1}^n w_i$, and the
multiplicity of the eigenvalue $\sum w_i$ with sum over $m$ terms is the number of fixed point free permutations of $N-m$ elements. Independently, \cite{fill96} derived a formula for the $m^{th}$ step transition matrix. Moreover, he measured the distance of the distribution of a random to top shuffle deck from its stationary distribution using the separation distance. \cite{donnelly} and \cite{kapoor} also determined these eigenvalues independently.

For top to random shuffles the spectral decomposition of the transition matrix has some algebraic interpretations. 
\cite{solomon} proved that $A_S= \sum_{D(\pi)=S} \pi$ forms an algebra, where $D(\pi)$ is the descent set of $\pi$. 
\cite{diaconis92} showed that, when $S=\left\{1\right\}$, $A_{\left\{1\right\}}$ generates an $n$-dimensional commutative semisimple subalgebra of the group
algebra, $\mathcal{A}(S_n)$ and $A_{\left\{1\right\}}$ has the same spectral decomposition as the transition matrix for top to random shuffles. \cite{garsia} also proved the same result.
\cite{aldous86}, \cite{diaconis92} and recently \cite{stark} proved that $n\log n+cn$ top to random shuffles is sufficient for the deck to be random.

In this paper we look at the following problem. Consider an ordered deck of $n$ cards labeled $1$ through $n$, where $n$ is an even integer. 
After $m$ top to random shuffles (we consider $w_i=\frac{1}{n}$ for all $i$), a person is asked to guess the cards starting with the top card. 
During the process the guesser receives no information, i.e., neither the type of the card is revealed nor if the guessed card was correct or not. 
We determine the best no feedback guessing strategy, i.e., the strategy that maximizes the expected number of correct guesses.
We prove that, given a deck of $n$ cards, for $n$ even, $c\geq 0$ and $m > 4n\log n+cn$, the best no feedback guessing strategy after $m$ top to random shuffles is to guess card $n-1$ for positions $1$ through $n/2$ and card $n$ for positions $n/2+1$ through $n$. 
By using this best no feedback card guessing strategy we show that after $m>4n\log n+cn$ top to random shuffles the deck is close to randomly distributed deck. 
\cite{ciucu98} studied no feedback card guessing for riffle shuffles.

\section{Calculation of the m step position matrix}
We start by considering the structure of the position matrix for top to random shuffles. The position matrix $P$ is an $n\times n$ matrix whose $(j,k)^{th}$ entry is
the probability that card $j$ moves to position $k$ after one shuffle.
\begin{lem}\label{positionmatrix}
For $1\leq j,k\leq n$, the position matrix has the form
\[
P_{jk} = \begin{cases}
  \frac{1}{n}       & \mbox{ if } j=1 \quad 1\leq k\leq n,
  \\
  \frac{n-(j-1)}{n} & \mbox{ if } j\ne 1 \quad k=j-1,
  \\
  \frac{j-1}{n}     & \mbox{ if } j\ne 1 \quad k=j,
  \\
  0                 & \mbox{ otherwise.}
\end{cases}
\]
\end{lem}
\begin{proof}
After one top to random shuffle card $1$ (the top card) can go to any position with equal probability $\frac{1}{n}$.
If card $1$ moves to any position below the $(j-1)^{th}$ position, card $j$ moves to the $(j-1)^{th}$ position.
Since there are $n-(j-1)$ such positions, this implies that card $j$
moves to position $j-1$ with probability $\frac{n-(j-1)}{n}$.
Otherwise, if card $1$ moves to any position above the position of the $j^{th}$ card, card $j$ stays in the same position.
This occurs with probability $\frac{j-1}{n}$.
Since these are the only possible positions for card $j$ after one shuffle, all other positions have zero probability.
\end{proof}
\begin{prop}\label{eigenvector}
Let $v_k{(j)}$ represent the $j^{th}$ component of the $k^{th}$ eigenvector of the position matrix $P$ and
$\lambda_k$ the corresponding eigenvalue for $v_k$. Then
\[
v_k(j) = \left\{ \begin{array}{cll}
(-1)^{n+j} \binom{n-k}{j-k}  & j\geq k        & 1\leq k\leq n-1,
\\
0                            & j<k            & 1\leq k\leq n-1,
\\
1                            & 1\leq j\leq n  & k=n,
\end{array} \right.
\]
and $\lambda_k = \frac{k-1}{n}$ for all $1\leq k\leq n-1$ and $\lambda_n = 1$.
\end{prop}
\begin{proof}
Assume that $j\geq k$ and $1\leq k\leq n-1$. We have
\[ (Pv_k)_{jk} = \frac{n-j+1}{n}v_k(j-1) + \frac{j-1}{n}v_k(j) = (-1)^{n+j} \frac{k-1}{n} \binom{n-k}{j-k} = \lambda_kv_k(j).\]
The other cases are trivial.
\end{proof}
\noindent Let $E$ be the $n\times n$ matrix whose $k^{th}$ column is $v_k$ for $1\leq k\leq n$.
By Proposition~\ref{eigenvector} it follows that $E^{-1}PE$ is a diagonal matrix $D$, where the $k^{th}$ diagonal element is $\lambda_k$,
the eigenvalue corresponding to $v_k$. To compute the probability of card $j$ being in position $k$ after a certain number of shuffles
we need the powers of the position matrix $P$.
\begin{lem}\label{inversematrix}
The inverse of $E$ is given by
\[
E^{-1}_{jk} = \left\{ \begin{array}{cll}
  v_k{(j)}-\frac{1}{n}\sum_{i=1}^j v_i{(j)}    & 1\leq j<n         & 1\leq k\leq j,
  \\
  \frac{-1}{n}\sum_{i=1}^j v_i{(j)}            & 1\leq j<n         & j<k\leq n,
  \\
  \frac{1}{n}                                  & j=n               & 1\leq k\leq n.
\end{array} \right.
\]
\end{lem}
\begin{proof}
It is enough to show $E^{-1}E = I$. Let $j=k$. We have
{\allowdisplaybreaks
\begin{align*}
  \left(E^{-1}E\right)_{jj} & = \sum_{i=1}^j(v_i(j)-\frac{1}{n}\sum_{k=1}^jv_k(j))v_j(i) + 	
  															\sum_{i=j+1}^n\left(-\frac{1}{n}\sum_{k=1}^jv_k(j)\right)v_j(i)
  \\
  & = \sum_{i=1}^jv_i(j)v_j(i)-\frac{1}{n}\sum_{k=1}^jv_k(j)\sum_{i=1}^nv_j(i) = v_j(j)v_j(j) = 1.
\end{align*}
}%
Similar steps show $\left (E^{-1}E\right)_{jk}=0$ for $j<k$ and $j>k$.
\end{proof}
\begin{cor}\label{pm}
The entries of $P^m$ are 
\[
  P^m_{jk} = \begin{cases}
  \frac{1}{n^{m+1}}\sum_{i=2}^j {(-1)^{j+i+1}(i-1)^m\binom{n-i}{j-i} \binom{n}{i-1}} + \frac{1}{n}     & 1\leq j<k \leq n,
  \\
  \frac{1}{n^{m+1}}\sum_{i=2}^j{(-1)^{j+i+1}(i-1)^m\binom{n-i}{j-i} \binom{n}{i-1}}                    &
  \\
  \qquad + \frac{1}{n^m}\sum_{i=k}^j{(-1)^{j+i}(i-1)^m\binom{n-i}{j-i} \binom{n-k}{i-k}} + \frac{1}{n} & 1\leq k\leq j\leq n.
  \end{cases}
\]
\end{cor}

\section{The Best Guessing Strategy for Top to Random Shuffles}
We now determine the best guessing strategy for top to random shuffles by using the structure of the $m^{th}$ power of the transition matrix $P$ derived in Corollary \ref{pm}. The final
result is given by Theorem \ref{best_strategy}.

\begin{prop}\label{last_row_increases}
Let $n\geq 4$ be even and $c\geq 0$. For $m > n\log n+cn$, $P^m_{nk}$ is increasing in $k$ for $1 \leq k\leq n$.
\end{prop}
\begin{proof}
Consider
\[
  P^m_{n(k+1)} - P^m_{nk} = \sum_{i=k}^{n-1}(-1)^{n+i+1}(i-1)^m\frac{n-i}{n-k}\binom{n-k}{i-k},
\]
and define $a_i := (i-1)^m\frac{n-i}{n-k}\binom{n-k}{i-k}$.
\noindent For $k\leq i\leq n-2$, the quotient satisfies
{\allowdisplaybreaks
\begin{align*}
  \frac{a_{i+1}}{a_i} &= \left(1+\frac{1}{i-1}\right)^m \frac{n-i-1}{i+1-k} 
  > \left(1+\frac{1}{n-3}\right)^m \frac{1}{n-1} > \left(1+\frac{1}{n-1}\right)^m\frac{1}{n} .
\end{align*}
}%
The first inequality holds, because 
  $\frac{n-i-1}{i+1-k} > \frac{n-i-1}{i+1} > \frac{1}{n-1}$.

\noindent Let $c \geq 0$. We know that $m\log(1+\frac{1}{n-1}) = m\sum_{t=1}^\infty\frac{1}{t}(\frac{1}{n})^t$. Hence, for $m > n\log n+cn$,
$m\log(1+\frac{1}{n-1}) > \log n + c$.  This implies $\frac{a_{i+1}}{a_i} > \left(1+\frac{1}{n-1}\right)^m \frac{1}{n} > e^c$.  Therefore, $a_i$ is increasing
for $m > n\log n+cn$ and $k\leq i\leq n-1$.  Hence, for all $m > n\log n+cn$, $P^m_{n(k+1)} - P^m_{nk}= \sum_{i=k}^{n-1}(-1)^{i+1}a_i > 0$.
\end{proof}
\begin{prop}\label{rows_decrease}
Let $n\geq 4$ be even and $c\geq 0$. For $m > n\log 2n+cn$ and $1\leq j\leq n-1$ fixed, $P^m_{jk}$ is decreasing in $k$ for $1\leq k\leq n$.
\end{prop}
\begin{proof}

Case 1: $k=j$
 
$P^m_{jj} - P^m_{j(j+1)} = \frac{(j-1)^m}{n^m} \geq 0 \text{ for all } m$.\\
Case 2: $k>j$ 

$P^m_{jk} - P^m_{j(k+1)} = 0 \text{ trivially for all }m.$\\
Case 3: $k<j$

 $ P^m_{jk} - P^m_{j(k+1)} = \frac{1}{n^m}\binom{n-k}{n-j}(-1)^{j+k}\sum_{i=0}^{j-k}(-1)^i(i+k-1)^m\binom{j-k}{i} \left(1-\frac{i}{n-k}\right)$.
 \\
Define 
\begin{equation}\label{eq:jdec}
  a_i : = (i+k-1)^m\binom{j-k}{i}\left(1-\frac{i}{n-k}\right)
\end{equation}
\noindent 
and let $c\geq 0$. We observe below that $a_i$, defined as in (\ref{eq:jdec}), is increasing for $m > n\log 2n+cn$.
For $0\leq i\leq j-k-1$,
{\allowdisplaybreaks
\begin{align*}
  \frac{a_{i+1}}{a_i}
 	& = \left(1+\frac{1}{i+k-1}\right)^m\frac{n-k-i-1}{n-k-i}\frac{j-k-i}{i+1}
	\\
  & \geq \left(1+\frac{1}{n-3}\right)^m \frac{1}{2(n-2)} > \left(1+\frac{1}{n-1}\right)^m \frac{1}{2n}.
\end{align*}
}%
The first inequality holds, since $i\leq j-k-1$, $j\leq n-1$, and $k\geq 1$. 
For $m > n\log 2n+cn$, we get 
$m\log (1+\frac{1}{n-1}) = m\sum_{t=1}^\infty\frac{1}{t}(\frac{1}{n})^t > \frac{m}{n} > \log 2n+c$.
Hence, $\frac{a_{i+1}}{a_i} \geq (1+\frac{1}{n-1})^m\frac{1}{2n} > e^c$ and so $a_i$ is increasing.
We conclude that for $m > n\log 2n+cn$, $\sum_{i=0}^{j-k}(-1)^{i+j+k} a_i >0$ and so $P^m_{jk} - P^m_{j(k+1)} > 0$.
\end{proof}
In the next lemma we state some properties of the binomials used in Proposition~\ref{last_row_max_left}, Proposition~\ref{last_row_max_right} and Corollary \ref{n_max_j}.
\begin{lem}\label{combo-lem-1}
Let $n\geq 10$ be even, then
\begin{itemize}
  \item[(i)] $\frac{1}{n}\binom{n}{i-1}-\binom{n/2-1}{i-1-n/2} > 0$ for all $\frac{n}{2}+1 \leq i\leq n-1$.
  \item[(ii)] $\frac{\frac{1}{n}\binom{n}{i}-\binom{n/2-1}{i-n/2}}{\frac{1}{n}\binom{n}{i-1}-\binom{n/2-1}{i-1-n/2}}
        \text{ is decreasing in $i$ for } \frac{n}{2}+1 \leq i\leq n-2$.
  \item[(iii)] $\frac{\frac{1}{n}\binom{n}{i}-\binom{n/2-1}{i-n/2}}{\frac{1}{n}\binom{n}{i-2}-\binom{n/2-1}{i-2-n/2}}
  \text{ is decreasing in } i \text{ for all } \frac{n}{2}+2 \leq i\leq n-2$.
  \item[(iv)] $\frac{\binom{n-i-1}{j-i-1}-1}{\binom{n-i}{j-i}-1} \text{ is decreasing in }i \text{ for }1 \leq i\leq j-2 \text{ and } 3 \leq j\leq n-1$.
\end{itemize}
\end{lem}
\begin{prop}\label{last_row_max_left}
Let $n\geq 10$ be even and $c\geq 0$. For $m > 3n\log n+cn$, we have $P^m_{n(\frac{n}{2}+1)} > P^m_{j(\frac{n}{2}+1)}$ for $1 \leq j\leq \frac{n}{2}$.
\end{prop}
\begin{proof}
We have
{\allowdisplaybreaks
\begin{align*}
  P^m_{n(\frac{n}{2}+1)} &- P^m_{j(\frac{n}{2}+1)} 
 \\
 & = \frac{1}{n^m}\left(\sum_{i=2}^j (-1)^i(i-1)^m \frac{1}{n}\binom{n}{i-1}\left((-1)^{j+2}\binom{n-i}{j-i} -1\right)\right.
  \\
  & + \sum_{i=j+1}^{n/2}(-1)^{n+i+1}(i-1)^m \frac{1}{n}\binom{n}{i-1}
  \\
  & + \left.\sum_{i=n/2+1}^{n-1}(-1)^{i+1}(i-1)^m \left(\frac{1}{n}\binom{n}{i-1}-\binom{n/2-1}{i-n/2-1}\right)\right).
\end{align*}
}%
\noindent Case 1: If $j$ is even, the equality above simplifies to
\[ \label{eq:lrml_even1}
  P^m_{n(\frac{n}{2}+1)}  - P^m_{j(\frac{n}{2}+1)}
  = \frac{1}{n^m}\sum_{i=2}^{n-2}(-1)^i a_i,
\]
where $a_i$ is defined as
\begin{equation}\label{eq:lrml_even}
a_i : = \begin{cases}
  (i-1)^m \frac{1}{n}\binom{n}{i-1}\left(\binom{n-i}{j-i}-1\right)	 & 2\leq i\leq j-1,
  \\
  i^m \frac{1}{n}\binom{n}{i} 										            & j\leq i\leq \frac{n}{2}-1,
  \\
  i^m\left(\frac{1}{n}\binom{n}{i}-\binom{n/2-1}{i-n/2}\right) & \frac{n}{2} \leq i\leq n-2.
\end{cases}
\end{equation}
\noindent Case 2: If $j$ is odd, the equality above becomes
\[
  P^m_{n(\frac{n}{2}+1)} - P^m_{j(\frac{n}{2}+1)} =
  \frac{1}{n^m}\sum_{i=1}^{n-2}(-1)^ia_{i},
\]
where $a_i$ is defined as
\begin{equation}\label{eq:lrml_odd}
  a_i = \begin{cases}
    i^m \frac{1}{n}\binom{n}{i}\left(\binom{n-i-1}{j-i-1}+1\right)           & 1\leq i\leq j-1,
    \\
    i^m \frac{1}{n}\binom{n}{i}                                          & j\leq i\leq \frac{n}{2}-1,
    \\
    i^m\left(\frac{1}{n}\binom{n}{i}-\binom{n/2-1}{i-n/2}\right) & \frac{n}{2} \leq i\leq n-2.
  \end{cases}
\end{equation}

\noindent Using equation (\ref{eq:lrml_even}) we have the following bounds for $\frac{a_{i+1}}{a_i}$.

\noindent Case 1.a: For $2 \leq i\leq j-2$,
{\allowdisplaybreaks
\begin{align*}
  \frac{a_{i+1}}{a_i} & =
  \left(1+\frac{1}{i-1}\right)^m\frac{n-i+1}{i}\frac{\binom{n-i-1}{j-i-1}-1}{\binom{n-i}{j-i}-1}
  \\
  & \geq \left(1+\frac{2}{n-8}\right)^m\frac{n+8}{n-6}\frac{2(n+2)}{(n+4)(n+1)}
  > \left(1+\frac{2}{n-8}\right)^m\frac{2}{n-6}
  \\
  & > \left(1+\frac{1}{n-1}\right)^m\frac{1}{n^3}.
\end{align*}
}%
The first inequality holds, since $j\leq \frac{n}{2}-1$ and by Lemma~\ref{combo-lem-1} part (iv) $\frac{\binom{n-i-1}{j-i-1}-1}{\binom{n-i}{j-i}-1} > \frac{\binom{n-j+1}{1}-1}{\binom{n-j+2}{2}-1}$.

\noindent Case 1.b: For $i=j-1$,
{\allowdisplaybreaks
\begin{align*}
  \frac{a_{i+1}}{a_i} &= \left(1+\frac{2}{j-2}\right)^m\frac{n-j+2}{j(j-1)}\frac{n-j+1}{n-j}
  \\
  & > \left(1+\frac{4}{n-6}\right)^m\frac{2}{n-2} > \left(1+\frac{1}{n-1}\right)^m\frac{1}{n^3}.
\end{align*}
}%
The first inequality follows, since $j\leq \frac{n}{2}-1$.

\noindent Case 1.c: For $j \leq i\leq \frac{n}{2}-2$,
\[
\frac{a_{i+1}}{a_i} 
= \left(1+\frac{1}{i}\right)^m\frac{n-i}{i+1}
> \left(1+\frac{2}{n-4}\right)^m > \left(1+\frac{1}{n-1}\right)^m\frac{1}{n^3}.
\]

\noindent Case 1.d: For $i=\frac{n}{2}-1$,
{\allowdisplaybreaks
\begin{align*}
  \frac{a_{i+1}}{a_i} & =
  \left(1+\frac{2}{n-2}\right)^m\left(\frac{n+2}{n}-\frac{n}{\binom{n}{n/2-1}}\right)
  \\
 & \geq \left(1+\frac{2}{n-2}\right)^m\frac{2}{n} > \left(1+\frac{1}{n-1}\right)^m\frac{1}{n^3}.
\end{align*}
}%
The first inequality holds, since for all $n\geq 6$, $\binom{n}{n/2-1} > \binom{n}{1}$. 

\noindent Case 1.e: For $\frac{n}{2} \leq i\leq n-3$, we have
{\allowdisplaybreaks
\begin{align*}
  \frac{a_{i+1}}{a_i} & =
   \left(1+\frac{1}{i}\right)^m\frac{\frac{1}{n}\binom{n}{i+1}
  -\binom{n/2-1}{i+1-n/2}}{\frac{1}{n}\binom{n}{i}-\binom{n/2-1}{i-n/2}}
  \geq \left(1+\frac{1}{n-3}\right)^m\frac{12}{(n-2)(n+8)} 
  \\
  & > \left(1+\frac{1}{n-1}\right)^m\frac{1}{n^3}.
\end{align*}
}%
The first inequality holds, by Lemma~\ref{combo-lem-1} part (ii) $\frac{\frac{1}{n}\binom{n}{i+1}-\binom{n/2-1}{i+1-n/2}}{\frac{1}{n}\binom{n}{i}-\binom{n/2-1}{i-n/2}}
\geq \frac{\frac{1}{n}\binom{n}{n-2}-\binom{n/2-1}{n/2-2}}{\frac{1}{n}\binom{n}{n-3}-\binom{n/2-1}{n/2-3}}$.

\noindent Using equation (\ref{eq:lrml_odd}) we have the following bounds for $\frac{a_{i+1}}{a_i}$. 

\noindent Case 2.a: For $1\leq i\leq j-2$, we have
{\allowdisplaybreaks
\begin{align*}
  \frac{a_{i+1}}{a_i} & =
  \left(1+\frac{1}{i}\right)^m\frac{n-i}{i+1} \frac{1+\binom{n-i-2}{j-i-2}}{1+\binom{n-i-1}{j-i-1}} 
 > \left(1+\frac{1}{i}\right)^m\frac{n-i}{n-i-1}\frac{j-i-1}{2(i+1)}
  \\
  & \geq \left(1+\frac{2}{{n-6}}\right)^m\frac{1}{n-4} > \left(1+\frac{1}{n-1}\right)^m\frac{1}{n^3}.
\end{align*}
}%
The first inequality holds, since 
$j\leq \frac{n}{2}-1$ and 
$\frac{1+\binom{n-i-2}{j-i-2}}{1+\binom{n-i-1}{j-i-1}} 
> \frac{\binom{n-i-2}{j-i-2}}{2\binom{n-i-1}{j-i-1}} $.
 The second inequality also follows, since $j\leq \frac{n}{2}-1$.

\noindent Case 2.b: For $i=j-1$, we have
{\allowdisplaybreaks
\begin{align*}
  \frac{a_{i+1}}{a_i} & 
  = \left(1+\frac{1}{j-1}\right)^m\frac{n-j+1}{2j}
  > \left(1+\frac{2}{{n-4}}\right)^m\frac{1}{2} > \left(1+\frac{1}{{n-1}}\right)^m\frac{1}{n^3}.
\end{align*}
}%
The first inequality follows, since $j\leq \frac{n}{2}-1$. 

\noindent Note that the remaining three cases, which are $j\leq i\leq \frac{n}{2}-2 \text{, } i=\frac{n}{2} -1$ and $\frac{n}{2} \leq i\leq n-3$, are identical to the cases 1.c, 1.d and 1.e.

\noindent Let $c\geq0$. 
For $m > 3n\log n+cn$, we get $\frac{m}{n} > 3\log n+c$ and, since $m\log(1+\frac{1}{n-1}) = m\sum_{t=1}^{\infty}\frac{1}{t}(\frac{1}{n})^t$, this implies $m\log(1+\frac{1}{n-1}) > 3\log n+c$. Hence, $(1+\frac{1}{n-1})^m\frac{1}{n^3} > e^c$.
Let us consider the ratio $\frac{a_{i+1}}{a_i}$ for both $j$ even and odd cases.
In all cases above we show $\frac{a_{i+1}}{a_i} > \left(1+\frac{1}{n-1}\right)^m\frac{1}{n^3}$.
Therefore, for $2 \leq i\leq n-2$, $a_{i}$, defined as in (\ref{eq:lrml_even}) and (\ref{eq:lrml_odd}) is increasing for $m > 3n\log n+cn$. We conclude that for $m > 3n\log n+cn$ and $j$ odd, $\sum_{i=1}^{n-2}(-1)^{i}a_i=(-a_1+a_2)+(-a_3+a_4)+ \ldots +(-a_{n-2}+a_{n-1}) >
0$. Similarly, for $j$ even, $\sum_{i=2}^{n-2}(-1)^i a_i = a_2 + \sum_{i=3}^{n-2}(-1)^i a_i > 0$. Hence, $P_{n(\frac{n}{2}+1)}^m-P_{j(\frac{n}{2}+1)}^m > 0$ for both $j$ even and odd.

\end{proof}
\begin{prop}\label{last_row_max_right}
  Let $n\geq 10$ be even and $c\geq 0$. For $m > 4n\log n+cn$, we have $P_{n(\frac{n}{2}+1)}^m > P_{j(\frac{n}{2}+1)}^m$ for $\frac{n}{2}+1\leq j\leq n-1$.
\end{prop}
\begin{proof}
Consider
{\allowdisplaybreaks
\begin{align*}
  P_{n(\frac{n}{2}+1)}^m &- P_{j(\frac{n}{2}+1)}^m 
 \\ & =\frac{1}{n^m}\left(\sum_{i=2}^{n/2}(-1)^i(i-1)^m\binom{n}{i-1}\frac{1}{n}\left((-1)^{j+2}\binom{n-i}{j-i}-1\right)\right.
  \\
  &+ \sum_{i=n/2+1}^j(-1)^i(i-1)^m\left(\frac{1}{n}\binom{n}{i-1}-\binom{n/2-1}{i-1-n/2}\right)
  \\
  & \times \left((-1)^{j+2}\binom{n-i}{j-i}-1\right)
  \\
  &+ \left.\sum_{i=j+1}^{n-1}(-1)^{i+1}(i-1)^m\left(\frac{1}{n}\binom{n}{i-1}-\binom{n/2-1}{i-1-n/2}\right)\right).
\end{align*}
}%
The proof for $P_{n(\frac{n}{2}+1)}^m - P_{j(\frac{n}{2}+1)}^m >0$ for $m > 4n \log n +cn$ and $c \geq 0$ is analog to the proof of Proposition~\ref{last_row_max_left}.
\end{proof}
\begin{cor}\label{n_max_j}
Let $n\geq 10$ be even and $c\geq 0$. For $m > 4n\log n+cn$ and $1\leq j\leq n-1$ fixed, we have $P^m_{nk} > P^m_{jk}$ for all $\frac{n}{2}+1\leq k\leq n$. 
\end{cor}
\begin{proof}
This result follows from Proposition~\ref{last_row_increases}, Proposition~\ref{rows_decrease}, Proposition~\ref{last_row_max_left} and Proposition~\ref{last_row_max_right}.
\end{proof}
The next lemma states some properties of the binomials used in Proposition~\ref{center_column_max_lastrow}, Proposition~\ref{decreasing_differences}, Proposition~\ref{center_column_max} and Corollary \ref{n-1_max_n}.
\begin{lem}\label{combo-lem-2}
Let $n\geq 8$ be even, then
\begin{itemize}
\item[(i)] $\frac{1}{n}\binom{n}{i-1}-\binom{n/2}{i-n/2} > 0 \text{ for } \frac{n}{2}\leq i\leq n-2$.
\item[(ii)] $\frac{\frac{1}{n}\binom{n}{i}-\binom{n/2}{i+1-n/2}}{\frac{1}{n}\binom{n}{i-1}-\binom{n/2}{i-n/2}}
\text{ is decreasing in $i$ for }\frac{n}{2} \leq i\leq n-3$.
\item[(iii)] $\frac{\frac{1}{n}\binom{n}{i}-\binom{n/2}{i+1-n/2}}{\frac{1}{n}\binom{n}{i-2}-\binom{n/2}{i-1-n/2}}
\text{ is decreasing in $i$ for }\frac{n}{2}+1 \leq i\leq n-3$.
\item[(iv)] $\frac{\binom{n-i-1}{j-i-1}-(n-i-1)}{\binom{n-i}{j-i}-(n-i)} \text{ is decreasing in $i$ for } 2 \leq i\leq j-3 \text{ and } 5 \leq j\leq n-2$.
\end{itemize}
\end{lem}

\begin{prop}\label{center_column_max_lastrow}
  Let $n\geq 8$ be even and $c\geq 0$. For $m > n\log n+ cn$, we have $P^m_{(n-1)\frac{n}{2}} > P^m_{n\frac{n}{2}}$.
\end{prop}
\begin{proof}
{\allowdisplaybreaks
\begin{align*}
  P^m_{(n-1)\frac{n}{2}} &- P^m_{n\frac{n}{2}} 
  \\
  & = \frac{1}{n^m} \left(\sum_{i=2}^{n/2-1}(-1)^i(i-1)^m\frac{n-i+1}{n}\binom{n}{i-1}\right.
  \\
  & + \left. \sum_{i=n/2}^{n-2}(-1)^i(i-1)^m(n-i+1) \left(\frac{1}{n}\binom{n}{i-1}-\binom{n/2}{i-n/2}\right) +(n-2)^m \right)
  \\
  & = \frac{1}{n^m}\left(\sum_{i=2}^{n-2}(-1)^ia_i+(n-2)^m\right),
\end{align*}
}%
where $a_i$ is defined as
\begin{equation}\label{eq:n-1vsnmiddle}
a_i=\begin{cases}
 (i-1)^m\frac{n-i+1}{n}\binom{n}{i-1} & 2\leq i\leq \frac{n}{2}-1,
 \\
 (i-1)^m(n-i+1)\left(\frac{1}{n}\binom{n}{i-1}-\binom{n/2}{i-n/2}\right) & \frac{n}{2} \leq i\leq n-2.
\end{cases}
\end{equation}

\noindent Using equation (\ref{eq:n-1vsnmiddle}) we get the following bounds for $\frac{a_{i+1}}{a_i}$.

\noindent Case 1.a: For $2\leq i\leq \frac{n}{2}-2$, we have
{\allowdisplaybreaks
\begin{align*}
  \frac{a_{i+1}}{a_{i}}& = 
   \left(1+\frac{1}{i-1}\right)^m\frac{n}{i}-1
  >\left(1+\frac{2}{n-6}\right)^m
  > \left(1+\frac{1}{n-1}\right)^m \frac{1}{n}.
\end{align*}
}%

\noindent Case 1.b: For $i=\frac{n}{2}-1$, we have
{\allowdisplaybreaks
\begin{align*}
  \frac{a_{i+1}}{a_i}
  & 
  = \left(1+\frac{2}{n-4}\right)^m\frac{n+2}{n+4}\frac{\frac{1}{n}\binom{n}{n/2-1}-1}{\frac{1}{n}\binom{n}{n/2-2}}
  >\left(1+\frac{2}{n-4}\right)^m \frac{1}{n}
  \\
  & > \left(1+\frac{1}{n-1}\right)^m \frac{1}{n}.
\end{align*}
}%
The first inequality follows from $\frac{\frac{1}{n}\binom{n}{n/2-1}-1}{\frac{1}{n}\binom{n}{n/2-2}} >
\frac{\binom{n}{n/2-1}}{\binom{n}{n/2-2}}-\frac{n}{\binom{n}{1}}= \frac{6}{(n-2)}$ for all $n\geq 6$.

\noindent Case 1.c: For $\frac{n}{2} \leq i\leq n-3$, we have
{\allowdisplaybreaks
\begin{align*}
  \frac{a_{i+1}}{a_i} &= \left(1+\frac{1}{i-1}\right)^m \frac{n-i}{n-i+1}
  \frac{\frac{1}{n}\binom{n}{i}-\binom{n/2}{i+1-n/2}}{\frac{1}{n}\binom{n}{i-1} -\binom{n/2}{i-n/2}}
  \\
	& >\left(1+\frac{1}{n-4}\right)^m \frac{3(n-4)}{2(n^2-4n+6)}
  >\left(1+\frac{1}{n-1}\right)^m \frac{1}{n}.
\end{align*}
}%
The first inequality holds by Lemma~\ref{combo-lem-2} part (ii),
$\frac{\frac{1}{n}\binom{n}{i}-\binom{n/2}{i+1-n/2}}{\frac{1}{n}\binom{n}{i-1} -\binom{n/2}{i-n/2}}
\geq \frac{\frac{1}{n}\binom{n}{n-3}-\binom{n/2}{n/2-2}}{\frac{1}{n}\binom{n}{n-4} -\binom{n/2}{n/2-3}}$. 
The second inequality holds for all $n\geq 6$. 

\noindent Let $c\geq 0$. For $m > n\log n+cn$, we get $\frac{m}{n} > \log n+c$ and, 
since $m\log(1+\frac{1}{n-1}) = m\sum_{t=1}^\infty\frac{1}{t}(\frac{1}{n})^t$ 
this implies $m\log(1+\frac{1}{n-1}) > \log n+c$. 
Hence, $(1+\frac{1}{n-1})^m \frac{1}{n} > e^c$. 
Above we prove that $\frac{a_{i+1}}{a_i} > \left(1+\frac{1}{n-1}\right)^m \frac{1}{n}$.
Therefore, for $2\leq i\leq n-2$, $a_i$, defined as in (\ref{eq:n-1vsnmiddle}) is increasing 
and so, $P^m_{(n-1)\frac{n}{2}}-P^m_{n\frac{n}{2}} > 0$ for all $m > n\log n+cn$.
\end{proof}
\begin{cor}\label{n-1_max_n}
  Let $n\geq 8$ be even and $c\geq 0$. For $m>n\log 2n+cn$, we have $P^m_{(n-1)k} > P^m_{nk}$ for $1\leq k\leq \frac{n}{2}$.
\end{cor}
\begin{proof}
This result follows from Proposition~\ref{last_row_increases}, Proposition~\ref{rows_decrease} and Proposition~\ref{center_column_max_lastrow}.
\end{proof}
\begin{prop}\label{decreasing_differences}
  Let $n\geq 8$ be even and $c\geq 0$. For $m > 2n\log n+cn$ and $1\leq j\leq n-2$ fixed, $P^m_{(n-1)k}-P^m_{jk}$ is decreasing in $k$ for $1\leq k\leq \frac{n}{2}$.
\end{prop}
\begin{proof}
\noindent Case 1: For $j > k$, we have
{\allowdisplaybreaks
\begin{align*}
  P^m_{(n-1)k}&-P^m_{(n-1)(k+1)}+P^m_{j(k+1)}-P^m_{jk} 
  \\
  & = (-1)^{k-1}\left(\frac{k-1}{n}\right)^m(n-k)-(-1)^{j+k}\left(\frac{k-1}{n}\right)^m\binom{n-k}{j-k}
  \\
  &+ \frac{1}{n^m}\left(\sum_{i=k+1}^j(i-1)^m\binom{n-k-1}{i-k}\right.
  \\
 & \times \left((-1)^{j+i+1}\binom{n-i}{j-i}+(-1)^{n-1+i}(n-i)\right)
  \\
  &+ \left.\sum_{i=j+1}^{n-1}(-1)^{n-1+i}(i-1)^m(n-i)\binom{n-k-1}{i-k}\right).
\end{align*}
}%
\noindent Case 1.1: For $j$ even, we have
\[
  P^m_{(n-1)k} - P^m_{(n-1)(k+1)}+P^m_{j(k+1)}-P^m_{jk} =
  \frac{1}{n^m}\sum_{i=k}^{n-1} (-1)^{i-1}a_i,
\]
where $a_i$ is defined as
\begin{equation}\label{eq:rowdiff_jgk_even}
a_i = \begin{cases}
  (i-1)^m\left(\binom{n-i}{j-i}+n-i\right)\binom{n-k-1}{i-k} & k\leq i\leq j,
  \\
  (i-1)^m(n-i)\binom{n-k-1}{i-k}  & j+1 \leq i\leq n-1.
\end{cases}
\end{equation}
\noindent Case 1.2: For $j$ odd, we have
\[
  P^m_{(n-1)k} - P^m_{(n-1)(k+1)}+P^m_{j(k+1)}-P^m_{jk} =
  \frac{1}{n^m}\sum_{i=k}^{n-2}(-1)^i a_i,
\]
where $a_i$ is defined as
\begin{equation}\label{eq:rowdiff_jgk_odd}
a_i = \begin{cases}
  (i-1)^m\left(\binom{n-i}{j-i}-n+i\right)\binom{n-k-1}{i-k} & k\leq i\leq j-2,
  \\
  i^m(n-i-2)\binom{n-k-1}{i+1-k} & i=j-1,
  \\
  i^m(n-i-1)\binom{n-k-1}{i+1-k}  & j\leq i\leq n-2.
\end{cases}
\end{equation}
\noindent Case 2: For $j=k$, we have
{\allowdisplaybreaks
\begin{align*}
  P^m_{(n-1)k} &- P^m_{(n-1)(k+1)}+P^m_{j(k+1)}-P^m_{jk} = \frac{1}{n^m}\left(((-1)^{n-1+k}(n-k)-1)(k-1)^m \right.
  \\
  &+ \left. \sum_{i=k+1}^{n-1}(-1)^{n-1+i}(i-1)^m(n-i)\binom{n-k-1}{i-k}\right).
\end{align*}
}%
\noindent Case 2.1: For $j$ even, we have
\[
  P^m_{(n-1)k} - P^m_{(n-1)(k+1)}+P^m_{j(k+1)}-P^m_{jk} =
  \frac{1}{n^m}\sum_{i=k}^{n-1}(-1)^{i-1} a_i,
\]
where $a_i$ is defined as
\begin{equation}\label{eq:rowdiff_jek_even}
a_i = \begin{cases}
  (i-1)^m(n-i+1) & i=k,
  \\
  (i-1)^m(n-i)\binom{n-k-1}{i-k}  & k+1 \leq i\leq n-1.
\end{cases}
\end{equation}

\noindent Case 2.2: For $j$ odd, we have
\[
  P^m_{(n-1)k} - P^m_{(n-1)(k+1)} + P^m_{j(k+1)} - P^m_{jk}
  = \frac{1}{n^m}\sum_{i=k}^{n-1}(-1)^{i-1} a_i,
\]
where $a_i$ is defined as
\begin{equation}\label{eq:rowdiff_jek_odd}
a_i = \begin{cases}
  (i-1)^m(n-i-1) & i=k,
  \\
  (i-1)^m(n-i)\binom{n-k-1}{i-k}  & k+1 \leq i\leq n-1.
\end{cases}
\end{equation}
\noindent Case 3: For $j < k$, we have
{\allowdisplaybreaks
\begin{align*}
	P^m_{(n-1)k} &- P^m_{(n-1)(k+1)} + P^m_{j(k+1)} - P^m_{jk} 
	\\ 
	& = \frac{1}{n^m}\sum_{i=k}^{n-1}(-1)^{i-1}(i-1)^m(n-i)\binom{n-k-1}{i-k} 
	= \frac{1}{n^m}\sum_{i=k}^{n-1}(-1)^{i-1}a_i,
\end{align*}
}%
where $a_i$ is defined as
\begin{equation}\label{eq:rowdiff_jlk}
a_i=(i-1)^m(n-i)\binom{n-k-1}{i-k}.
\end{equation}

\noindent Using equation (\ref{eq:rowdiff_jgk_even}) we get the following cases for $\frac{a_{i+1}}{a_i}$.

\noindent Case 1.1.a: For $k\leq i\leq j-1$, we have
{\allowdisplaybreaks
\begin{align*}
  \frac{a_{i+1}}{a_i} 
  & = \left(1+\frac{1}{i-1}\right)^m\frac{n-i-1}{i-k+1} \frac{\binom{n-i-1}{j-i-1}+n-i-1}{\binom{n-i}{j-i}+n-i}
  \\
  & > \left(1+\frac{1}{i-1}\right)^m\frac{(n-i-1)(j-i)}{2(i-k+1)(n-i)}
  \\
  & >\left(1+\frac{1}{n-4}\right)^m\frac{1}{3(n-3)} > \left(1+\frac{1}{n-1}\right)^m\frac{1}{n^2}.
\end{align*}
}%
The first inequality holds, since $\frac{\binom{n-i-1}{j-i-1}+n-i-1}{\binom{n-i}{j-i}+n-i} > \frac{\binom{n-i-1}{j-i-1}}{2\binom{n-i}{j-i}} = \frac{j-i}{2(n-i)}$.

\noindent The second inequality holds, since $k\leq i\leq j-1$, $i-k+1\leq j-k$ and, since $j\leq n-2$ and $k\geq 1$,we get $j-k\leq n-3$. Hence, $\frac{j-i}{i-k+1}\geq \frac{1}{n-3}$.

\noindent Case 1.1.b: For $i=j$, we have
{\allowdisplaybreaks
\begin{align*}
  \frac{a_{i+1}}{a_i} & 
  = \left(1+\frac{1}{j-1}\right)^m\frac{n-j-1}{n-j+1}\frac{n-j-1}{j-k+1}\\
  & > \left(1+\frac{1}{n-3}\right)^m\frac{1}{3(n-2)} 
  > \left(1+\frac{1}{n-1}\right)^m\frac{1}{n^2}.
\end{align*}
}%

\noindent Case 1.1.c: For $j+1 \leq i\leq n-2$, we have
{\allowdisplaybreaks
\begin{align*}
  \frac{a_{i+1}}{a_i} & 
  = \left(1+\frac{1}{i-1}\right)^m \frac{n-i-1}{n-i}\frac{n-i-1}{i-k+1}
  \\
  & > \left(1+\frac{1}{n-3}\right)^m\frac{1}{2(n-2)} 
  > \left(1+\frac{1}{n-1}\right)^m\frac{1}{n^2}.
\end{align*}
}%

\noindent Using equation (\ref{eq:rowdiff_jgk_odd}) we get the following cases for $\frac{a_{i+1}}{a_i}$.

\noindent Case 1.2.a: For $k\leq i\leq j-3$, we have
{\allowdisplaybreaks
\begin{align*}
  \frac{a_{i+1}}{a_i} & =
  \left(1+\frac{1}{i-1}\right)^m\frac{n-i-1}{i-k+1}\frac{\binom{n-i-1}{j-i-1}-(n-i-1)}{\binom{n-i}{j-i}-(n-i)}
  \\
	& > \left(1+\frac{1}{j-4}\right)^m\frac{n-j+2}{j-k-2}\frac{3(n-j+2)}{(n-j+3)(n-j+4)}
	\\
  & \geq \left(1+\frac{1}{n-6}\right)^m\frac{8}{5(n-5)} > \left(1+\frac{1}{n-1}\right)^m\frac{1}{n^2}.
\end{align*}
}%
The first inequality holds, since Lemma~\ref{combo-lem-2} part (iv) implies,

$\frac{\binom{n-i-1}{j-i-1}-(n-i-1)}{\binom{n-i}{j-i}-(n-i)}
    \geq \frac{\binom{n-j+2}{2}-(n-j+2)}{\binom{n-i}{3}-(n-j+3)} $.
 \noindent The second inequality holds, since $j\leq n-2$.

\noindent Case 1.2.b: For $i=j-2$, we have
{\allowdisplaybreaks
\begin{align*}
  \frac{a_{i+1}}{a_i} & 
  = \left(1+\frac{2}{j-3}\right)^m\frac{2(n-j+1)(n-j)}{(j-k)(j-k-1)(n-j+2)}
  \\
  & > \left(1+\frac{2}{n-5}\right)^m\frac{3}{(n-3)(n-4)} > \left(1+\frac{1}{n-1}\right)^m\frac{1}{n^2}.
\end{align*}
}%
The first inequality holds, since $j\leq n-2$.

\noindent Case 1.2.c: For $i=j-1$, we have
{\allowdisplaybreaks
\begin{align*}
  \frac{a_{i+1}}{a_i} & 
  = \left(1+\frac{1}{j-1}\right)^m\frac{n-j-1}{j-k+1}
  \\
  & \geq \left(1+\frac{1}{n-3}\right)^m\frac{1}{n-2} > \left(1+\frac{1}{n-1}\right)^m\frac{1}{n^2}.
\end{align*}
}%
The first inequality holds, since $j\leq n-2$ and $k\geq1$.

\noindent Case 1.2.d: For $j \leq i\leq n-3$, the case is identical to case 1.1.3.

\noindent Using equation (\ref{eq:rowdiff_jek_even}) we get the following cases for $\frac{a_{i+1}}{a_i}$.

\noindent Case 2.1.a: For $i=k$, we have
{\allowdisplaybreaks
\begin{align*}
  \frac{a_{i+1}}{a_i}
  & 
  = \left(1+\frac{1}{k-1}\right)^m\frac{(n-k-1)^2}{n-k+1}
  \\
  & > \left(1+\frac{2}{n-2}\right)^m\frac{n^2-4}{2(n+2)} > \left(1+\frac{1}{n-1}\right)^m\frac{1}{n^2}.
\end{align*}
}%
The first inequality holds, since $k\leq \frac{n}{2}$.

\noindent Case 2.1.b: For $k+1 \leq i\leq n-2$, we have

{\allowdisplaybreaks
\begin{align*}
  \frac{a_{i+1}}{a_i} & 
  = \left(1+\frac{1}{i-1}\right)^m\frac{(n-i-1)^2}{(n-i)(i+1-k)}
  \\
  & > \left(1+\frac{1}{n-2}\right)^m\frac{1}{2}\frac{1}{n-1} > \left(1+\frac{1}{n-1}\right)^m\frac{1}{n^2}.
\end{align*}
}%

\noindent Using equation (\ref{eq:rowdiff_jek_odd}) we get the following cases for $\frac{a_{i+1}}{a_i}$.

\noindent Case 2.2.a: For $i=k$, we have
{\allowdisplaybreaks
\begin{align*}
  \frac{a_{i+1}}{a_i} &= \left(1+\frac{1}{k-1}\right)^m(n-k-1)
  \\
  & \geq \left(1+\frac{2}{n-2}\right)^m\frac{n-2}{2} > \left(1+\frac{1}{n-1}\right)^m\frac{1}{n^2}.
\end{align*}
}%
The first inequality holds, since $k\leq \frac{n}{2}$.

\noindent Case 2.2.b: For $k+1 \leq i\leq n-2$, the case is identical to case 2.1.b.

\noindent Using equation (\ref{eq:rowdiff_jlk}) we get the following cases for $\frac{a_{i+1}}{a_i}$.

\noindent Case 3.1: For $k\leq i\leq n-2$, the cases for $j$ even and odd are the same.
{\allowdisplaybreaks
\begin{align*}
  \frac{a_{i+1}}{a_i}
  & 
  = \left(1+\frac{1}{i-1}\right)^m\frac{(n-i-1)^2}{(n-i)(i+1-k)}
  \\
  & > \left(1+\frac{1}{n-2}\right)^m\frac{1}{2(n-1)} > \left(1+\frac{1}{n-1}\right)^m\frac{1}{n^2}.
\end{align*}
}%

\noindent Let $c\geq 0$. For $m > 2n\log n+cn$, we get $\frac{m}{n} > 2 \log n +c $ and since 
$m\log(1+\frac{1}{n-1}) = m\sum_{t=1}^\infty\frac{1}{t}(\frac{1}{n})^t$
this implies $m\log(1+\frac{1}{n-1}) > 2\log n+c$. 
Hence, $\left(1+\frac{1}{n-1}\right)^m\frac{1}{n^2} > e^c$.
In all cases above we show that $\frac{a_{i+1}}{a_i} > \left(1+\frac{1}{n-1}\right)^m\frac{1}{n^2}$. Therefore, $a_i$, defined as in (\ref{eq:rowdiff_jgk_even}) and (\ref{eq:rowdiff_jgk_odd}) for $j > k$,
(\ref{eq:rowdiff_jek_even}) and (\ref{eq:rowdiff_jek_odd}) for $j=k$ and finally (\ref{eq:rowdiff_jlk}) for $j<k$, is increasing 
and hence,
$P^m_{(n-1)k} - P^m_{(n-1)(k+1)} + P^m_{j(k+1)} - P^m_{jk} =\sum_{i=k}^{n-1}(-1)^{i-1}a_i > 0$, for all $m>2n\log n+cn$.

\end{proof}
\begin{prop}\label{center_column_max}
  Let $n\geq 10$ be even and $c\geq 0$. For $m > 3n\log n+cn$, we have $P^m_{(n-1)\frac{n}{2}} > P^m_{j\frac{n}{2}}$ for $1\leq j\leq n-2$.
\end{prop}
\begin{proof}
\noindent Case 1: Let $j < \frac{n}{2}$. 
{\allowdisplaybreaks
\begin{align*}
  P_{(n-1)\frac{n}{2}}-P_{j\frac{n}{2}}
  &= \frac{1}{n^m}\left(\sum_{i=2}^j(-1)^i(i-1)^m \frac{1}{n}\binom{n}{i-1}\left(n-i+(-1)^j\binom{n-i}{j-i}\right)\right.
  \\
  &+ \sum_{i=j+1}^{n/2-1}(-1)^i(i-1)^m\frac{n-i}{n}\binom{n}{i-1}
  \\
  &+ \left. \sum_{i=n/2}^{n-1}(-1)^i(i-1)^m(n-i)\left(\frac{1}{n}\binom{n}{i-1}-\binom{n/2}{i-n/2}\right)\right).
\end{align*}
}%
\noindent Case 2: Let $j = \frac{n}{2}$. 
{\allowdisplaybreaks
\begin{align*}
 & P^m_{(n-1)\frac{n}{2}} - P^m_{\frac{n}{2}\frac{n}{2}} 
  \\
  & = \frac{1}{n^m}\left(\sum_{i=2}^{n/2-1}(-1)^{i+1}(i-1)^m \frac{1}{n}\binom{n}{i-1}\left((-1)^{(n/2+1)}\binom{n-i}{n/2-i}-(n-i)\right)\right.
  \\
  & + (-1)^{n/2}\left(\frac{n}{2}-1\right)^m (\frac{n}{2}+(-1)^{n/2})\left(\frac{1}{n}\binom{n}{n/2-1}-1\right)
  \\
  & + \left.\sum_{i=n/2}^{n-3}(-1)^{i+1}i^m(n-(i+1))\left(\frac{1}{n}\binom{n}{i}-\binom{n/2}{i+1-n/2}\right)+\frac{(n-2)^m}{2}\right)
  \\
  & = \frac{1}{n^m}\left(\sum_{i=2}^{n-3}(-1)^{i+1}a_i+\frac{(n-2)^m}{2}\right).
\end{align*}
}%
\noindent Case 3: Let $j > \frac{n}{2}$. 
{\allowdisplaybreaks
\begin{align*}
  & P^m_{(n-1)\frac{n}{2}} - P^m_{j\frac{n}{2}}
  \\
  & = \frac{1}{n^m} \left(\sum_{i=2}^{n/2-1}(-1)^i(i-1)^m \frac{1}{n}\binom{n}{i-1}\left(n-i+(-1)^j\binom{n-i}{j-i}\right)\right.
  \\
  &+ \sum_{i=n/2}^j(-1)^i(i-1)^m \left( \frac{1}{n}\binom{n}{i-1}-\binom{n/2}{i-n/2} \right) \left(n-i+(-1)^j\binom{n-i}{j-i} \right)
  \\
  &+ \left.\sum_{i=j+1}^{n-1}(-1)^i(i-1)^m(n-i) \left(\frac{1}{n}\binom{n}{i-1}-\binom{n/2}{i-n/2} \right)\right).
\end{align*}
}%
Parallel to the proof of Proposition~\ref{decreasing_differences} we consider both even and odd subcases to prove that $P^m_{(n-1)\frac{n}{2}} - P^m_{j\frac{n}{2}}> 0$ for $m > 3n\log n+cn$ and $c \geq 0$.
\end{proof}
\begin{cor}\label{n-1_max_j}
  Let $n\geq 10$ be even and $c\geq 0$. For $m > 3n\log n+cn$ and $1\leq j\leq n-2$ fixed, we have $P^m_{(n-1)k}\geq P^m_{jk}$ for $1\leq k\leq \frac{n}{2}$.
\end{cor}
\begin{proof}
The proof follows from Proposition~\ref{decreasing_differences} and Proposition~\ref{center_column_max}.
\end{proof}

\begin{thm}\label{best_strategy}
Let $n\geq 10$ be even and $c\geq 0$. For $m>4n\log n+cn$, the best no feedback guessing strategy is to guess card $n-1$ for positions 1 through $n/2$ and card $n$ for
positions $n/2+1$ through $n$.
\end{thm}
\begin{proof}
The proof follows from Corollaries~\ref{n_max_j}, \ref{n-1_max_n}, and \ref{n-1_max_j}.
\end{proof}
\section{Convergence To Uniformity}
In this section we use Theorem \ref{best_strategy} to determine the number of top to random shuffles needed for a deck of cards to become uniformly distributed,
i.e., all orderings are equally likely. Here, the distance used to measure the difference between a top to random shuffled deck and a uniform deck is given by the difference of the average number of correct guesses using the best no feedback guessing strategy and that of the uniform distribution. Note that without feedback the expected number of correct guesses for a uniformly distributed deck using any strategy is $1$. 
\begin{thm}\label{expectation}
Let $n\geq 10$ be even, $c\geq 0$ and $m > 4n\log n+cn$. Further, let $E^m(n)$ denote the expected number of correct guesses using the best no feedback guessing strategy for a deck of $n$ cards after $m$ top to random shuffles.
Then $E^m(n)-1 \leq e^{-c}$.
\end{thm}
\begin{proof}

By Theorem~\ref{best_strategy}, $E^m(n)-1 = \sum_{i=1}^{n/2} P^m_{(n-1)i}+ \sum_{i=n/2+1}^nP^m_{ni}-1$ for all $m>4n\log n+cn$.
{\allowdisplaybreaks
\begin{align*}
  & \sum_{i=1}^{n/2} P^m_{(n-1)i} + \sum_{i=n/2+1}^nP^m_{ni}-1
  \leq \frac{n}{2}P^m_{(n-1)1} + \frac{n}{2}P^m_{nn}-1
  \\
  & < \frac{n}{2} \left( \frac{1}{n^m}\sum_{i=1}^{n-1} {(-1)^{n-1+i}(i-1)^m (n-i)\binom{n-1}{i-1}} \right.
  \\
  & \qquad + \left. \frac{1}{n^{m+1}}\sum_{i=2}^{n-1}(-1)^{n+1+i}(i-1)^m \binom{n}{i-1} + \frac{2}{n}\right) -1
  \\
  & = \frac{1}{2n^{m-1}} \left((n-1)!S(m,n-2) + (n-1)!S(m,n) + (n-1)^m - n^{m-1} \right)
  \\
  & < \frac{1}{2n^{m-1}} \left(\sum_{k=0}^m S(m,k)(n-1)_k +\frac{1}{n}\sum_{k=0}^m S(m,k)(n)_k+ (n-1)^m - n^{m-1} \right)
  \\
  & = \frac{1}{2n^{m-1}} \left((n-1)^m + \frac{n^m}{n} + (n-1)^m - n^{m-1} \right)
  \\
  & = n \left( 1-\frac{1}{n} \right)^m < ne^{-\frac{1}{n}m} < ne^{-\frac{1}{n}(4n\ln(n)+cn)} \leq e^{-c}.
\end{align*}
}%
The first inequality holds by Proposition~\ref{last_row_increases} and Proposition~\ref{rows_decrease}.
The second inequality holds, since the first term in $P^m_{(n-1)1}$, namely $ \sum_{i=2}^{n-1} (-1)^{n+i} (i-1)^m(n-i)\binom{n}{i-1} = n! \left( S(m,n-1)+S(m,n)
\right) - n^m < 0$, where $S(m,n)$ is the Stirling number of the second kind. Similarly, the second term in $P^m_{nn}$, namely $\sum_{i=1}^{n-1}(-1)^{n+i}(i-1)^m \binom{n-1}{i-1} = (n-1)! S(m,n-1)- (n-1)^m = (n-1)! S(m,n-1) - \sum_{k=0}^m S(m,k)(n-1)_k < 0$.
\end{proof}
\begin{rem}
Theorem \ref{expectation} indicates that $4n\log n+cn$ shuffles are enough to mix the deck. This is in line with the results of \cite{aldous86} as well as \cite{diaconis92}. 
\end{rem}

\section*{Acknowledgements}
I would like to thank Jason Fulman for helpful discussions and Daniel Panario for his useful comments. 

\bibliographystyle{apalike}
\bibliography{pehlivan-nfb-card-guessing-new}

\begin{thebibliography}{}

\bibitem[Aldous and Diaconis, 1986]{aldous86}
Aldous, D. and Diaconis, P. (1986).
\newblock Shuffling cards and stopping times.
\newblock {\em The American Mathematical Monthly}, 93(5):333--348.

\bibitem[Ciucu, 1998]{ciucu98}
Ciucu, M. (1998).
\newblock No-feedback card guessing for dovetail shuffles.
\newblock {\em Annals of Applied Probability}, 8(4):1251--1269.

\bibitem[Diaconis et~al., 1992]{diaconis92}
Diaconis, P., Fill, J.~A., and Pitman, J. (1992).
\newblock Analysis of top to random shuffles.
\newblock {\em Combinatorics, Probability and Computing}, 1:135--155.

\bibitem[Donnelly, 1991]{donnelly}
Donnelly, P. (1991).
\newblock The heaps process, libraries, and size-biased permutations.
\newblock {\em Journal of Applied Probability}, 28(2):321--335.

\bibitem[Fill, 1996]{fill96}
Fill, J.~A. (1996).
\newblock An exact formula for the move-to-front rule for self-organizing
  lists.
\newblock {\em Journal of Theoretical Probability}, 9(1):113--160.

\bibitem[Garsia and Wallach, 2007]{garsia}
Garsia, A.~M. and Wallach, N.~R. (2007).
\newblock r-qsym is free over sym.
\newblock {\em Journal of Combinatorial Theory, Series A}, 114(4):704--732.

\bibitem[Kapoor and Reingold, 1991]{kapoor}
Kapoor, S. and Reingold, E.~M. (1991).
\newblock Stochastic rearrangement rules for self-organizing data structures.
\newblock {\em Algorithmica}, 6(1-6):278--291.

\bibitem[McCabe, 1965]{mccabe}
McCabe, J. (1965).
\newblock On serial files with relocatable records.
\newblock {\em Operations Research}, 13(4):609--618.

\bibitem[Phatarfod, 1991]{phatar}
Phatarfod, R.~M. (1991).
\newblock On the matrix occurring in a linear search problem.
\newblock {\em Journal of Applied Probability}, 28(2):336--346.

\bibitem[Phatarfod, 1994]{phatar2}
Phatarfod, R.~M. (1994).
\newblock On the transition probabilities of the move-to-front scheme.
\newblock {\em Journal of Applied Probability}, 31(2):570--574.

\bibitem[Solomon, 1976]{solomon}
Solomon, L. (1976).
\newblock A {M}ackey formula in the group ring of a {C}oxeter group.
\newblock {\em Journal of Algebra}, 41(2):255--264.

\bibitem[Stark, 2002]{stark}
Stark, D. (2002).
\newblock Information loss in top to random shuffling.
\newblock {\em Combinatorics, Probability and Computing}, 11(6):607--627.

\end{thebibliography}
\hfill Carleton University

\hfill School of Mathematics and Statistics

\hfill Ottawa, ON Canada K1S 5B6

\hfill pehlivan@math.carleton.edu

\end{document}